\documentclass[12pt]{amsart}

\usepackage{amssymb}
\usepackage{amsmath,amsthm}
\usepackage{fullpage}
\usepackage{xcolor}
\usepackage{mathrsfs}

\usepackage[colorlinks = true, citecolor = blue, pagebackref = false]{hyperref} % hyperlinks cross-references
\newcommand{\C}{\mathbb{C} }
\newcommand{\R}{\mathbb{R} }
\newcommand{\Z}{\mathbb{Z} }
\newcommand{\T}{\mathbb{T} }
\newcommand{\Q}{\mathbb{Q} }
\newcommand{\N}{\mathbb{N} }
\newcommand{\eof}[1]{e \inparentheses{#1}}
\newcommand{\FT}[1]{\mathcal{F}_{{#1}}}

\newcommand{\inparentheses}[1]{\left( #1 \right)}

\renewcommand{\vector}[1]{{\boldsymbol{#1}}}
\newcommand{\form}{\mathcal{Q}}
\newcommand{\kernel}{K}

\newcommand{\dd}{\; {\rm d}}

\newcommand{\dm}{\; \mathrm{d}m}

\newcommand{\dalpha}{\; \mathrm{d}\alpha}

\newcommand{\major}{\mathfrak{M}}		% major arcs
\newcommand{\frakM}{\mathfrak{M}}		% major arcs
\newcommand{\frakm}{\mathfrak{m}}		% minor arcs
\DeclareMathOperator{\dilate}{D}
\newtheorem{proposition}{Proposition}
\newtheorem{lemma}{Lemma}
\newtheorem{theorem}{Theorem}

\newtheorem{conjecture}{Conjecture}

\newtheorem*{Magyar}{Magyar's Decomposition Theorem}

\newtheorem*{Schmidt}{Schmidt's Theorem}

\theoremstyle{definition}

\theoremstyle{remark}
\newtheorem{remark}{Remark}[section] % For the reader/paper.
\title{Supercritical discrete restriction estimates for forms in many variables}

\author{Brian Cook}
\address{
    Department of Mathematics
	\\	Virginia Tech
	\\	McBryde Hall
	\\	225 Stanger Street
	\\	Blacksburg, VA 24061-1026
	\\	USA
}
\email{briancookmath@gmail.com}

\author{Kevin Hughes}
\address{
    School of Mathematics
	\\	The University of Bristol
	\\	Fry Building
	\\	Woodland Road
	\\	Bristol, BS8 1UG
	\\	UK
	\\ and the Heilbronn Insitute for Mathematical Research, Bristol, UK
}
\email{khughes.math@gmail.com}

\author{Eyvindur Palsson}
\address{
    Department of Mathematics
	\\	Virginia Tech
	\\	McBryde Hall
	\\	225 Stanger Street
	\\	Blacksburg, VA 24061-1026
	\\	USA
}
\email{palsson@vt.edu}

\begin{document}
\maketitle
%\tableofcontents

\begin{abstract}
We prove discrete restriction estimates for a broad class of hypersurfaces and varieties of intermediate codimension. 
For our result about hypersurfaces, we use Bourgain's arithmetic version of the Tomas--Stein method and Magyar's decomposition of Birch's hypersurfaces. 
For our result about varieties of higher codimension, we use the even moment method and works of Birch and Schmidt. 
%Insert `2000 Mathematics Subject Classification' numbers here:
%\subclass{42B25 (primary) \and 11P05 \and 11P55 \and 11L07 \and 37A45 (secondary)}
%  Please refer to {\tt http://www.ams.org/msc/} for a list of codes
\end{abstract}

% ---
\section{Introduction}
% ---

In this paper, we consider discrete restriction estimates associated to integral, positive definite, homogeneous forms; let $\form(\vector{x}) \in \Z[\vector{x}]$ be such a homogeneous form, where $\vector{x} = (x_1,x_2,\dots,x_d)$ with $d \geq 2$ and $k$ denotes the degree of the form $\form$. 
We assume that $k \geq 2$. 
For each $\lambda \in \R$, the polynomial $\form$ cuts out a real variety $V_{\form=\lambda}(\R) := \{\vector{x} \in \R^d : \form(\vector{x}) = \lambda\}$ containing a discrete set of integral points $V_{\form=\lambda}(\Z) := \{\vector{x} \in \Z^d : \form(\vector{x}) = \lambda\}$; either or both of these sets are possibly empty depending on the value of $\lambda$. 
For instance, since $\form$ is positive definite, $V_\form(\R)$ is empty for negative $\lambda$. 
Also, $V_{\form=\lambda}(\Z)$ is always empty for non-integral values of $\lambda$. 
%Note that for each $\lambda \in \N$, $N_{\form}(\lambda)$ is finite because the hypersurface $\{ \vector{n} \in \R^d : \form(\vector{n}) = \lambda \} $ is defined by a positive definite form which implies that this hypersurface is compact. 

Momentarily fix the form $\form$ so that we may suppress it from the notation below. 
% Fix a smooth, compactly supported function $\psi : \R^d \to \C$. 
For $\lambda \in \N$ and functions $a : \Z^d \to \C$, define the arithmetic extension operator 
\[
E_\lambda a(\vector{\xi}) 
:= 
\sum_{\vector{x} \in V_{\form=\lambda}(\Z^d)} a(\vector{x}) \eof{\vector{x} \cdot \vector{\xi}}.
\]
Letting $\omega_\lambda := {\bf 1}_{ V_{\form=\lambda}(\Z^d)}$, 
% denote the un-normalized arithmetic surface measure of $V_\form(\Z)$, 
we have 
\(
E_\lambda a(\vector{\xi}) 
= 
\FT{\Z^d}({a\cdot \omega_\lambda})(\vector{\xi})
\) 
where $\FT{\Z^d}$ is the Fourier transform defined on complex-valued functions with domain $\Z^d$. 
In other words, $E_\lambda$ is the adjoint to the restriction operator $R_{\lambda}f$ defined as 
\[
f \mapsto R_{\lambda}f := \FT{\T^d}({f}) \cdot \omega_\lambda
\]
for functions $f : \T^d \to \C$. 
The extension operator is trivial when the variety has no integer points; that is, when $V_{\form=\lambda}(\Z^d)$ is the empty set. 
Consequently, we are interested in situations where the variety has many integer points. 
Due to a theorem of Birch, there is a natural setting for these operators which we now review. 

Define the \emph{Birch singular locus} of the form $\form$ as the complex variety 
\[ V_\form^\dagger(\C) := \{\vector{x} \in \C^d :  \nabla \form(\vector{x}) = \vector{0}\}. \]
Let $\dim_\C(V)$ denote the algebraic dimension of a complex variety $V$. 
We will say that a homogeneous, integral form is \emph{regular} if it satisfies Birch's criterion:
\begin{equation}\label{Birch_criterion}
d-\dim_\C(V_\form^\dagger(\C)) > (k-1)2^k.
\end{equation}
When \eqref{Birch_criterion} is satisfied, Birch \cite{Birch} tells us that there exists an infinite arithmetic progression $\Gamma_{\form}$ in $\N$ depending on the form $\form$ such that for each $\lambda \in \Gamma_{\form}$, there exists a positive constant $C_\form(\lambda)$ with the property that 
\begin{equation}\label{asymptotic:Birch}
N_{\form}(\lambda) 
:= \#\{ \vector{n} \in \Z^d : \form(\vector{n}) = \lambda \} 
= C_{\form}(\lambda) \lambda^{\frac{d}{k}-1} + O_\form(\lambda^{\frac{d}{k}-1-\delta})
> 0;
%\quad \text{for all} \quad \lambda \in \Gamma_{\form},
\end{equation} 
for some positive $\delta$ depending on the form $\form$. 
Moreover, there exists constants $c_2>c_1>0$ such that $c_1 \leq C_\form(\lambda) \leq c_2$ for all $\lambda \in \Gamma_{\form}$. 
% Following Magyar \cite{Magyar:ergodic}, we will call any such arithmetic progression $\Gamma_\form$ a \emph{set of regular values for $\form$}. 
Based on Birch's asymptotic \eqref{asymptotic:Birch} and on the usual heuristics of the circle method one expects the following estimates. 
\begin{conjecture}\label{conjecture1}
Let $\form$ be an integral, positive definite, homogeneous form of degree $k \geq 2$ in $d > 2k$ variables. 
For each $1 \leq p \leq \infty$ and $\epsilon>0$, there exists a positive constant $C_{\form,p,\epsilon}$ such that 
\begin{equation}\label{estimate:conjecture:quadratic}
\|E_\lambda a\|_{L^p(\T^d)} 
\leq C_{\form,p,\epsilon}
\lambda^\epsilon (1+\lambda^{\frac{d-k}{2k}-\frac{d}{kp}}) \|a\|_{\ell^2(\Z^d)}.
\end{equation}
For $k \geq 3$ we further conjecture that one may remove the $\epsilon$-loss; that is, for each $1 \leq p \leq \infty$, there exists a constant $C_{\form,p}$ such that 
\begin{equation}\label{estimate:conjecture:higher}
\|E_\lambda a\|_{L^p(\T^d)} 
\leq C_{\form,p}
(1+\lambda^{\frac{d-k}{2k}-\frac{d}{kp}}) \|a\|_{\ell^2(\Z^d)}.
\end{equation}
\end{conjecture}
\noindent 
The critical exponent here is $p_c := \frac{2d}{d-k}$ where the two summands in \eqref{estimate:conjecture:quadratic} or \eqref{estimate:conjecture:higher} balance. 

There are two trivial estimates for Conjecture~\ref{conjecture1}. 
The first trivial estimate is the $\ell^2 \to L^2$ estimate which is furnished by Plancherel's theorem. 
The second trivial estimate is the $\ell^2 \to L^\infty$ estimate furnished by the Cauchy--Schwarz inequality and \eqref{asymptotic:Birch} when the latter is known to hold. 
Conjecture~\ref{conjecture1} has been intensively studied in the quadratic case, especially for the spherical case $\form(\vector{x}) := x_1^2+\cdots+x_d^2$. 
Even for the sphere, this problem remains open despite major recent advances in the area. 
See \cite{Bourgain:eigenfunctions,Bourgain:lattice_point_problems,Bourgain:moment_inequalities,BD:improved_sphere,BD:sphere_4D5D,BD} for more information regarding the spherical case and \cite{BD:Hypersurf} for other quadratic hypersurfaces. 
However, for a general form $\form$ of higher degree, there are no hitherto known nontrivial estimates towards this problem. 

In connection with results for discrete restriction to the sphere, A. Magyar (personal communication) asked how to incorporate minor arc estimates for higher degree Diophantine equations in order to obtain discrete restriction estimates. 
Our first result answers Magyar's question for a broad class of hypersurfaces and yields \eqref{estimate:conjecture:higher} when $p$ and $d$ are both sufficiently large. 
In particular, $p$ will be much larger than the critical exponent $p_c$. 
To state our first result, we introduce a relevant parameter. 
For a regular, homogeneous, integral form $\form$ of degree $k$ in $d$ variables define the parameter
\begin{equation*}%\label{eq:Birch:parameters}
\gamma_{\form} := \frac{1}{6k} \left( \frac{d-\dim(V_\form^\dagger(\C))}{(k-1)2^k} - 1 \right)
\end{equation*}
% and 
% \[
% \kappa_{\form} := \frac{d-\dim V_\form(\C)}{2^{k-1}(k-1)}.
% \]
%\begin{align*}
%& \alpha_{\form} := {\frac{d}{k}-1},
%\\ & \beta_{\form} := {\frac{d-2}{k}},
%\\ & \gamma_{\form} := \frac{1}{6k} \left( \frac{d-\dim(V_\form(\C))}{(k-1)2^k} - 1 \right),
%\\ & \kappa_{\form} := \frac{d-\dim V_\form(\C)}{2^{k-1}(k-1)} \quad \text{and} \quad
%\\ & \eta_{\form,p} := \min\{ \beta_\form(\frac{2}{p}-1) , \alpha_\form(\frac{2}{p}-1)+\gamma_\form(2-\frac{2}{p}) \}
%\quad \text{for} \quad 1<p<2.
%\end{align*}
% \begin{equation}\label{eq:Birch:parameters}
% \gamma_{\form} := \frac{1}{6k} \left( \frac{d-\dim(V_\form(\C))}{(k-1)2^k} - 1 \right)
% \quad \text{and} \quad
% \kappa_{\form} := \frac{d-\dim V_\form(\C)}{2^{k-1}(k-1)}.
% \end{equation}
%The Birch--Magyar non-singularity criterion \eqref{Birch_criterion} implies that $\gamma_\form>0$ and $\kappa_\form>2$. 
%Throughout we assume that $d > k \geq 2$ so that $\beta_\form \geq \alpha_\form>0$. 
Throughout we assume that $d$ is sufficiently large with respect to $k$ to satisfy the Birch--Magyar regularity criterion \eqref{Birch_criterion} with $\gamma_\form>0$. 
In particular, this implies that $d>2k$. 
\begin{theorem}\label{theorem:main}
% Let $\form$ be a regular, positive definite, homogeneous, integral form in $d$ variables of degree $k$ and $\Gamma_\form$ a set of regular values for $\form$. 
% If $p > 2+\frac{2k}{\gamma_\form}$, then \eqref{estimate:conjecture:higher} holds for $\lambda \in \Gamma_{\form}$. 
Let $\form$ be a regular, positive definite, homogeneous, integral form in $d$ variables of degree $k \geq 2$. 
If $p > 2+\frac{2k}{\gamma_\form}$, then \eqref{estimate:conjecture:higher} holds for $\lambda \in \N$. 
\end{theorem}

Recent progress in the spherical case has come in part by new, sharp decoupling estimates starting with \cite{BD}. 
These techniques have been significantly expanded upon to handle Diophantine equations with an affine invariant structure; affine invariance is also known as translation-dilation invariance or as parabolic rescaling. 
It is important to note the setting of Theorem~\ref{theorem:main} is far from affine invariant. 

Our second result concerns a similar class of operators associated to  surfaces which take the form of a graph over $\mathbb{Q}$. 
More specifically, these operators are defined for surfaces which take the shape $\{ (\vector{n},\form^{(2)}(\vector{n}),\dots,\form^{(k)}(\vector{n})) : \vector{n} \in \Z^d \}$ where, for $2 \leq i \leq k$, $\form^{(i)}$ is a complex-linearly independent collection of $r_i$ elements of $\Z[x_1,...,x_d]$ of degree exactly $i$. 
% The linear independence is over $\C$. 
The cutoff parameter is modified in a natural way for the operators associated to these surface: fix a function $\psi : \R^d \to \C$ which is Schwartz or continuous and compactly supported, and for $N>0$ define the truncated extension operator $E_N$ acting on functions $a : \Z^d \to \C$ as 
\[
E_N a(\vector{\xi}_1,\vector{\xi}_2,\dots,\vector{\xi}_k) 
:= 
\sum_{\vector{n} \in \Z^d} \psi(\vector{n}/N) a(\vector{n}) \eof{\vector{n} \cdot \vector{\xi}_1 + \form^{(2)}(\vector{n}) \cdot \vector{\xi}_2 + \cdots + \form^{(k)}(\vector{n}) \cdot \vector{\xi}_k }.
\]

As above, $\ell^2 \to L^p$ estimates for the extension operators $E_N$ have been intensively studied in the affine invariant setting. 
In particular, see \cite{BDG, GLY, Li, Wooley:cubic, Wooley:Restr, Wooley:nested} for the moment curve $(t,t^2,\dots,t^k)$ and related non-degenerate curves. 
For results in higher dimensions, see 
\cite{BD:surfaces_4D, BDGuo, DGS, GORYZ, Guo, GLY, GZ, GZk:ACK_systems, GZk:low_codimensions, HH:paraboloids, HH:quadratic, Oh, PPW}. 
Few results are known outside the affine invariant setting; see \cite{HuLi:Degree3, HW, LaiDing}. 
Current methods are often restricted to the case of quadratic or cubic equations of low dimension or low codimension. 
Higher degree equations present numerous new difficulties; see \cite{BrandesParsell:1,BrandesParsell:2,BHB} for a few recent advances on closely related arithmetic problems. 

We would like to formulate a conjecture for our present situation analogous to Conjecture~\ref{conjecture1}. 
The above works show that a general formulation is likely to be rather intricate. 
However, based upon these works, it seems plausible that there exists an exponent $p_0(\form) > 2$, depending on the form $\form$, such that the inequality 
\begin{equation}\label{estimate:conjecture2}
\|E_Na\|_{L^p(\T^{d+R})} 
\lesssim N^{\frac{d}{2}-\frac{D_\form}{p}} \|a\|_{\ell^2(\Z^d)}
\end{equation}
holds for $p>p_0(\form)$. 
Such an inequality would be sharp since the factor $N^{d/2-D_\form/p}$ in \eqref{estimate:conjecture2} arises as a lower bound to this problem by taking $a \equiv 1$ and counting the number of lattice points. 

In formulating our second result there are a few relevant parameters associated to $\form$ to consider. 
The first parameter is the  total number of polynomials in our system of forms; in our setup, this is $d+R_\form$ where we define $R_\form := r_2+...+r_k$ as the total number of nonlinear polynomials in our system. 
The second parameter is the \emph{total degree} of the system of forms $\form$; this is the quantity $d+2r_2+...+kr_k$, and is denoted as $D_\form$. %, or simply $D$ when no clarity on $\form$ is needed. 
Our third parameter is $r_\form=\sup_{\{2\leq i \leq k\}}r_i$. 

An additional notion that we will need is that of a type of rank for systems of polynomials. 
This is, more precisely, something which is defined in a graded fashion; a family of ranks is defined, one for each  degree. 
With $\form$ given,  define the ranks for $2\leq i \leq k$, 
\(
\mathcal{B}_i(\form)=\mathcal{B}_i(\form^{(i)})
\)
for each $i=1,...,k$, where 
\[
\mathcal{B}_i(\form^{(i)})=\{\vector{n}\in\C^d:\text{rank}(\text{Jac}_{\form_{\text{hom}}}^{(i)}(\vector{n}))<r_i\}
\]
and $\form_{\text{hom}}^{(i)}$ is top degree homogeneous part(s) of the $\form^{(i)}$. 
In the case $r_i=0$ for some $2\leq i \leq k$, we take $\mathcal{B}_i(\form^{(i)})=\infty$. Also, in the case that $\form$ is system of polynomials of common degree we will drop the subscript $i$.

Our second result is the following theorem. 
\begin{theorem}\label{theorem:images}
For each $i=2,\dots,k,$, let $\form_{hom}^{(i)}$ be a complex linearly independent collection of $r_i$ polynomials consisting of degree precisely $i$ in the $d$ variables $\vector{n} = (n_1,\dots,n_d)$. 
Let $\form(\vector{n})=(\vector{n},\form^{(2)}(\vector{n}),...,\form^{(k)}(\vector{n}))$. 
There exists constants $\chi(i,r_i,R,k)$ such that if $s\geq 2$ is an integer satisfying the inequalities 
\begin{equation}\label{Birch:even}
\mathcal{B}_i(\form) 
\geq \frac{\chi(i,r_i,R,k)+(r_i+1)d}{(r_i+s)}
\end{equation}
for $2\leq i \leq k$, then 
%{\color{red}Conjecture~\ref{conj2}} 
\eqref{estimate:conjecture2} holds when $p\geq 2(r_\form+s)$.
\end{theorem}

An explanation of the conditions in Theorem~\ref{theorem:images} is in order. 
A priori the only exponents $p$ for which \eqref{estimate:conjecture2} is known are the exponents $p=2$ and $\infty$; these are obtained by Plancherel's theorem and the Cauchy--Schwarz inequality respectively. 
Consequently the claimed region of $p \geq 2(r+s)$ is non-trivial for any system $(\vector{n},\form^{(2)}(\vector{n}),\dots,\form^{(k)}(\vector{n}))$. 
For instance, taking $s$ sufficiently large so that the right hand side of \eqref{Birch:even} is at most one, we obtain a non-trivial bound for systems of rank 1. 
In contrast, note that fixing the parameters $r,k,s$ and $p$ forces a lower bound on the minimal allowable dimension in Theorem~\ref{theorem:images}; in particular, we find that we require 
\[
d \geq \frac{\chi(i,r_i,R,k)}{(s-1)}.
\]

In terms of the constants $\chi$, we can take 
\[ 
\chi(i,r_i,R,k) = (i-1) 2^{3i} r_i R k. 
\] 
These are slightly weaker than those provided by \cite{Schmidt}, although hopefully more transparent. 
Schmidt's work built on work of Birch, which concerned the case when $r_i>0$ for exactly one $i\geq 2$. 
Birch's bounds in this case are those given above. 
A recent paper of Browning and Heath-Brown \cite{BHB} also approaches this question. 
The bounds given in \cite{Schmidt} do not reduce to those of \cite{Birch} stated above; the work of Browning and Heath-Brown bridges this gap.

% --
\subsection{Organization of the paper}
% --

The paper is organized as follows. 
Section~\ref{section:notation} sets some notation used throughout the paper. 
In Section~\ref{section:even_moment} we prove  Theorem~\ref{theorem:images} using the even moment method in \cite{Zygmund, Bourgain:discrete_restriction:NLS} and Schmidt's theorem in \cite{Schmidt}. 
In Section~\ref{section:preliminaries} we give an abstract formulation of the Tomas's method for discrete $L^2$ restriction theorems introduced in \cite{Bourgain:Squares, Bourgain:discrete_restriction:NLS} which reduces our problem to proving estimates related to the Fourier transform of the surface measure. 
In  Section~\ref{section:BirchMagyar} we recall a decomposition of the surface measure due to Magyar as well as relevant estimates from \cite{Magyar:ergodic}. 
In Section~\ref{section:majorarcs} we prove a bound for the major arc pieces by a further application of Thomas's methods and deduce Theorem~\ref{theorem:main} from this bound and the bounds of Section~\ref{section:BirchMagyar}.

% --
\subsection*{Acknowledgements}
% --

The authors would like to thank the Heilbronn Institute for Mathematical Research for enabling this collaboration through their Focused Research Workshop ``Efficient Congruencing and Decoupling". KH thanks Virginia Tech for their hospitality whilst part of this paper was written. EP was supported in part by Simons Foundation Grant \#360560.

% ---
\section{Notation}\label{section:notation}
% ---

We introduce here some notation that will streamline our exposition. 
\begin{itemize}
\item 
For a positive integer, we let $\mathbb{Z}/q$ denote the group of integers modulo $q$ and $U_q := \{ 1 \leq a < q : (a,q)=1 \}$ denote its unit group. 
\item
We write $f(\lambda) \lesssim g(\lambda)$ if there exists a constant $C>0$ independent of all $\lambda$ under consideration (e.g. $\lambda$ in $\N$ or in $\Gamma_{\form}$) such that 
\[
|f(\lambda)| \leq C |g(\lambda)|. 
\]
Furthermore, we will write $f(\lambda) \gtrsim g(\lambda)$ if $g(\lambda) \lesssim f(\lambda)$ while we will write $f(\lambda) \eqsim g(\lambda)$ if $f(\lambda) \lesssim g(\lambda)$ and $f(\lambda) \gtrsim g(\lambda)$.
%\item
%There will be several $\epsilon$-losses that arise. 
%Rather than writing ``for all $\epsilon>0$ ...'' we will write $f(\lambda) \lessapprox g(\lambda)$ if $f(\lambda) \lesssim_\epsilon \lambda^{\epsilon} g(\lambda)$ for all $\epsilon>0$; the implicit constants are allowed to depend on $\epsilon$ but not on $\lambda$. 
%For instance the conclusion of Theorem~\ref{theorem:discrete_improving:Birch} may be more succintly written as 
%\[
%\|A_\lambda f\|_{\ell^{p'}(\Z^d)} 
%\lesssim_{\epsilon,\form,p} \lambda^{-\eta_{\form,p}+\epsilon} \|f\|_{\ell^{p}(\Z^d)}.
%\]
% \item
Subscripts in the above notations will denote parameters, such as the dimension $d$ or degree $k$ of a form $\form$, on which the implicit constants may depend. 
\item
$\T^d$ denotes the $d$-dimensional torus $(\R/\Z)^d$ identified with the unit cube $[-1/2,1/2]^d$. 
\item
$*$ denotes convolution on a group such as $\Z^d$, $\T^d$ or $\R^d$. 
It will be clear from context as to which group the convolution takes place.
\item
$\eof{t}$ will denote the character $e^{-2\pi it}$ for $t \in \R$ or $\T$. 
\item
For a function $f: \Z^d \to \C$, its $\Z^d$-Fourier transform will be denoted $\FT{\Z^d}{f}(\vector{\xi})$ for $\vector{\xi} \in \T^d$. 
For a function $f: \T^d \to \C$, its $\T^d$-Fourier transform will be denoted $\FT{\T^d}{f}(\vector{x})$ for $\vector{x} \in \Z^d$. 
$\FT{\Z^d}$ and $\FT{\T^d}$ are defined so that they are inverses of one another. 
% \item 
For a function $f: \R^d \to \C$, its $\R^d$-Fourier transform will be denoted $\FT{\R^d}{f}(\vector{x})$ for $\vector{x} \in \R^d$. 
\item
For a function $f: \R^d \to \C$, we define dilation operator $\dilate_{t}$ by $\dilate_{t}f(\vector{x}) = f(\vector{x}/t)$. 
\item
For a ring $R$, we will use the inner product notation $\vector{b} \cdot \vector{m}$ for vectors $\vector{b},\vector{m} \in R^d$ to mean the sum $\sum_{i=1}^d b_i m_i$. 
This is used for the rings $\R,\Z,\T$ and $\Z/q$ where $q \in \N$. 
\item
We also let ${\bf 1}_X$ denote the indicator function of the set $X$. 
\end{itemize}

% ---
\section{The even moment method and forms in many variables}
\label{section:even_moment}
% ---

In this section we give a proof of Theorem \ref{theorem:images} using the even moment method. 
This method is classical, and in connection with restriction problems, it appears at least as far back as \cite{Zygmund}. 
The additional ideas introduced here come from the study of forms in many variables, and a central result in this area is the following theorem of Schmidt. 
\begin{Schmidt}[\cite{Schmidt}]\label{Schmidt}
Let $\form=(\form^{(2)},...,\form^{(k)}) \subset \Z[\vector{x}]$ be a family of forms where $\form^{(i)}$ consists of $r_i$ polynomials of degree $i$. 
Let $D=2r_2+...+kr_k$ denote the total degree of the family $\form$. 
There exist constants $\chi(i, r_i, k, R)$ such that if 
\(
{B}_i(\form) \geq \chi(i, r_i, k, R),
\)
then 
\begin{equation}\label{estimate:Schmidt}
% \sum_{\mathbf{n}\in\Z^d, |n_i| \leq N}\mathbf{1}_{\form(\mathbf{n})=\mathbf{t}} 
\#\{ \vector{n} \in \Z^d : |n_i| \leq N \; \text{for each} \; i \; \text{and} \; \form(\vector{n})=\vector{t} \}
\lesssim_{\form} N^{d-D_\form}
\end{equation}
uniformly in $\vector{t}\in\Z^R$. 
% The implied constant above is allowed to depend on $\form$, $d$, and $R$.
\end{Schmidt}

We need two further results. 
The first result that we will make use of is the following. 
\begin{lemma}\label{lemma1}
Let $k \geq 1$ and let $\form$ be a system of $r$ forms in $\Z[x_1,... , x_d]$, each of degree $k$. 
If $H$ is an affine linear space of co-dimension $m$, then the restriction of $\form$ to $H$ satisfies 
\[ \mathcal{B}(\form|_H)\geq \mathcal{B}(\form)-m(r+1). \]
\end{lemma}
\noindent This lemma originated in \cite{dioprimes} (Corollary 4, Section 4). 
A minor oversight of the first author is corrected in \cite{shuntaro} (Corollary 3.3, Section 3); this corrects the $r$ appearing in \cite{dioprimes} to the  $r+1$ appearing here.

We now describe the second result that we will make use of. 
For a family of forms $\form$, consider the polynomials $\form^{s\oplus}:\Z^{ds}\to\Z^R$, given by 
\[
\form^{s\oplus}(\vector{n}^{(1)},\dots, \vector{n}^{(s)})=\form(\vector{n}^{(1)})+\cdots+\form(\vector{n}^{(s)})
\]
for integral $s\geq1$. 
Specifically, the following lemma gives a comparison of 
$\mathcal{B}_i(\form)$ with that of $\mathcal{B}_i(\form^{s\oplus})$. 

\begin{lemma}\label{lemma2}
Let $\form$ be family of $r$ homogeneous polynomials of common degree $k\geq2$. If the family is linearly independent over $\C$, then $\mathcal{B} (\form^{s\oplus}) \geq s \cdot \mathcal{B} (\form) \geq s$.
\end{lemma}

\begin{proof}[Proof of Lemma~\ref{lemma2}]
The Jacobian for $\form^{s\oplus}$ has the block matrix structure
\[
\text{Jac}_{\form^{s\oplus}}(\vector{n}^{(1)},\dots, \vector{n}^{(s)}) 
= 
\left[\text{Jac}_\form(\vector{n}^{(1)}) \,|\,\text{Jac}_\form(\vector{n}^{(2)}) \,|\, \cdots \,|\, \text{Jac}_\form(\vector{n}^{(s)})\right] ,
\]
giving 
\[
V^\dagger_{\form^{s\oplus}}\subset \underbrace{V^\dagger_\form\times ...\times V^\dagger_{\form}}_{s\,\text{copies}}.
\]
In turn we see that 
$\text{dim}_{\C^{sd}} (V^\dagger_{\mathbf{Q}^{s\oplus}})
\leq s\,\text{dim}_{\C^d} (V^\dagger_{\mathbf{Q}})$, 
so  that 
\[
\mathcal{B}(\form) 
\geq 
sd-s\,\text{dim}_{\C^d}(V^\dagger_{\mathbf{Q}})=s\mathcal{B}(\form). 
\]
The condition on linear independence implies that $\mathcal{B}(\form) \geq 1$.  
\end{proof}

We are now prepared to prove Theorem~\ref{theorem:images}.
\begin{proof}[Proof of Theorem~\ref{theorem:images}]
Fix the degree $k \geq 2$ as well as the systems of polynomials $\form^{(2)},...,\form^{(k)}$. 
Define $\form(\vector{n})=(\vector{n},\form^{(2)}(\vector{n}), ...\form^{(k)}(\vector{n}))$. 
Suppose that $a \in \ell^2(\Z^d)$ and fix $l \in \N$. 
Writing out $\|E_Na\|_{L^{2l}(\T^{d+R})}$, foliating and applying the Cauchy--Schwarz inequality, we obtain 
\begin{align*}
\| E_Na \|_{L^{2l}(\T^{d+R})} 
& = 
\int_{\T^{d+R}} \left| \sum_{\vector{n}\in\Z^d} \psi(\vector{n}/N)a(\vector{n}) \eof{\form(\mathbf{n})\cdot\alpha} \right|^{2l} \,d\alpha
\\ & = 
\sum_{\vector{t}\in\Z^{d+R}} \left| \sum_{\vector{n}^{(1)},\dots,\vector{n}^{(l)}\in\Z^d} \left( \prod_{i=1}^l \psi(\vector{n}^{(i)}/N) a(\vector{n}^{(i)}) \right) \mathbf{1}_{\form^{s\oplus}(\vector{n}^{(1)},\dots,\vector{n}^{(l)})=\vector{t}} \right|^2
\\ & \lesssim 
\sum_{\vector{t}\in\Z^{d+R}}\left(  \sum_{\vector{n}^{(1)},\dots,\vector{n}^{(l)}\in\Z^d} \left| \prod_{i=1}^l a(\vector{n}^{(i)}) \right|^2 \mathbf{1}_{\form^{l\oplus}(\vector{n}^{(1)},...,\vector{n}^{(l)})=\vector{t}} \right) 
\\& \quad \quad \quad 
\times \left( \sum_{\vector{n}^{(1)},\dots,\vector{n}^{(k)} \in \Z^d} \left( \prod_{i=1}^l \psi(\vector{n}^{(i)}/N) \right) \mathbf{1}_{\form^{l\oplus}(\vector{n}^{(1)},...,\vector{n}^{(l)})=\vector{t}} \right).
\end{align*}
Therefore, 
\begin{align*}
\| E_Na \|_{L^{2l}(\T^{d+R})} 
& \lesssim 
\left( \sup_{\vector{t}\in\Z^{d+R}} \sum_{\vector{n}^{(1)},\dots,\vector{n}^{(l)} \in \Z^d} \psi(\vector{n}^{1}/N) \cdots \psi(\vector{n}^{l}/N) \mathbf{1}_{\form^{l\oplus}(\vector{n}^{(1)},...,\vector{n}^{(l)})=\vector{t}}\right)
\\& \quad \quad \quad \times 
\left( \sum_{\vector{t}\in\Z^{d+R}} \sum_{\vector{n}^{(1)},\dots,\vector{n}^{(k)}\in\Z^d}|a(\vector{n}^{(1)}) \cdots a(\vector{n}^{(l)})|^2 \mathbf{1}_{\form^{l\oplus}(\vector{n}^{(1)},\dots,\vector{n}^{(l)})=\vector{t}} \right)
\\ & \leq 
\left( \sup_{\vector{t}\in \Z^{d+R}} \sum_{\vector{n}^{(1)},\dots,\vector{n}^{(l)}\in\Z^d} \psi(\vector{n}^{1}/N) \cdots \psi(\vector{n}^{l}/N) \mathbf{1}_{\form^{l\oplus}(\vector{n}^{(1)},\dots,\vector{n}^{(l)})=\vector{t}}\right) \|a\|_{\ell^2(\Z^d)}^{2l}.
\end{align*}
Now we need a bound of the form 
\[
\sum_{\vector{n}^{(1)},\dots,\vector{n}^{(l)} \in \Z^d} \psi(\vector{n}^{1}/N) \cdots \psi(\vector{n}^{l}/N) \mathbf{1}_{\form^{l\oplus}(\vector{n}^{(1)},\dots,\vector{n}^{(l)})=\vector{t}}
\lesssim 
N^{ld-D_\form},
\] 
which is uniform over $\mathbf{t} \in \Z^{d+R}$, in order to obtain the desired estimate \eqref{estimate:conjecture2}. 
% \[ 
% \|E_Na\|_{L^p(\T^{R+d})} \lesssim N^{n/2-D_\form/2}\|a\|_{\ell^2(\Z^d)}. \]

This is, up to the appearance of the linear forms, the content of Schmidt's Theorem. 
We can remove the effect of the linear pieces with Lemma \ref{lemma1}. 
With $l$ fixed, set $H(\vector{t})$ to be the affine linear space given by the $d$ equations 
\[
\vector{n}^{(1)}+\cdots+\vector{n}^{(l)} = \vector{t}.
\]
The co-dimension of $H(\vector{t})$ is then $d$ for any $\vector{t}$. Now, by applying Lemma \ref{lemma1}, it follows that 
\[
\mathcal{B}_i(\form^{l\oplus}|_{H(\vector{t})})
\geq \mathcal{B}_i(\form^{l\oplus})-(r_i+1)d
\geq l\mathcal{B}_i(\form)-(r_i+1)d.
\]
Provided that $l\mathcal{B}_i(\form)\geq\chi(i, r_i, k, R)+(r_i+1)d$ for each $2\leq i\leq k$, and recalling that $r=\max_{2\leq i \leq k} r_i$, the result now follows from Schmidt's Theorem and setting $l=r+s$ with $s\geq 2$.
\end{proof}

% ---
\section{The arithmetic Tomas--Stein method}
\label{section:preliminaries}
% ---

% Let $\omega$ be arithmetic surface measure; that is, a measure supported on the integral points of some rational algebraic set. 
% For instance, in the context of Theorem~\ref{theorem:main}, 
Let $\omega_\lambda$ be the counting measure on $V_{\form=\lambda}$ for a single integral, positive definite, homogenous form $\form$ satisfying \eqref{Birch_criterion} and some $\lambda \in \Z$. 
Let $F = \FT{\Z^d}(\omega_\lambda)$ be the exponential sum corresponding to $\omega_\lambda$. 
A common approach to problems involving $\omega_\lambda$ is to use the circle method so as to decompose the exponential sum $F$ into a main piece $F_\frakM$ and an error term $F_\frakm$ corresponding respectively to major and minor arcs. 
To prove discrete restriction estimates, Bourgain, in \cite{Bourgain:discrete_restriction:NLS,Bourgain:discrete_restriction:KdV} combined this approach with the Tomas $L^2$ restriction argument in order to reduce matters to the following two estimates: 
\begin{itemize}
\item
prove bounds for the operator given by convolution with the major arc operator $F_\frakM$, 
\item
and prove a uniform power saving bound on the minor arc piece $F_\frakm$. 
\end{itemize}
See \cite{HuLi} for a close variant. 
Bourgain's approach has been abstracted in \cite{Keil} and \cite{HH:paraboloids}. 
We combine Lemmas~2 and 3 from \cite{HH:paraboloids} to form the following lemma. 
\begin{lemma}
\label{lemma:prelims:TomasSteinDcp}
%Let $N = \lambda^{1/k}$. 
% Let $N>0$ and $F : \T^d \to \C$ be a function such that $\|F\|_{L^\infty(\T^d)} \lesssim N^{d-k}$. 
For $\lambda \in \N$, let $F = \FT{\Z^d}(\omega_\lambda)$ be the $\Z^d$-Fourier transform of the arithmetic surface measure $\omega_\lambda$ defined on $V_{\form=\lambda}(\Z)$. 
Suppose that there exists a decomposition $F = F_{\frakM} + F_{\frakm}$ such that for each $f \in L^\infty(\T^d)$, we have the estimates 
% \begin{enumerate} 
% \item[(\ref{lemma:prelims:TomasSteinDcp}.1)]
% $\| F \ast f \|_{L^{p_0}(\T^r)} \lesssim N^{\epsilon} \|f\|_{L^{p_0'}(\T^d)}$
% for some $p_0 \leq p_c$.
% \item[(\ref{lemma:prelims:TomasSteinDcp}.2)]
% $\| F_{\frakM} \ast f \|_{p_1} \lesssim N^{{d-k}-\frac{2d}{p_1}} \|f\|_{p_1'}$
% for some $p_1 > p_c$, and 
% \item[(\ref{lemma:prelims:TomasSteinDcp}.3)]
% $\| F_{\frakm} \|_{\infty} \lesssim N^{d-k-\zeta}$ for some $\zeta \in (0,d-k)$. 
% \end{enumerate}
\begin{align} 
\label{estimate:subcritical}
& \| F \ast f \|_{L^{p_0}(\T^r)} \lesssim \lambda^{\epsilon} \|f\|_{L^{p_0'}(\T^d)}
\quad \text{for some} \quad 
p_0 \leq p_c, 
\tag{TS1}
\\ 
\label{estimate:majorarcs}
& \| F_{\frakM} \ast f \|_{p_1} \lesssim \lambda^{{\frac{d}{k}-1}-\frac{2d}{kp_1}} \|f\|_{p_1'}
\quad \text{for some} \quad p_1 > p_c, 
\; \text{and} 
\tag{TS2}
\\ 
\label{estimate:minorarcs}
& \| F_{\frakm} \|_{\infty} \lesssim N^{\frac{d}{k}-1-\frac{\zeta}{k}} 
\quad \text{for some} \quad \zeta \in (0,d-k). 
\tag{TS3}
\end{align}
Then
$\| F \ast f \|_{L^p(\T^r)} \lesssim \lambda^{\frac{d}{k}-1 - \frac{2d}{kp}} \| f \|_{L^{p'}}$
holds for $p > \max[p_1, \frac{2d - (d-k)p_0}{\zeta} + p_0 ]$. 
\end{lemma}

In our work, we only use Plancherel's theorem to exploit the subcritical estimate at $p_0 = 2$; this gives the exponent $p > \max[\, p_1, \frac{2d - (d-k)2}{\zeta} + 2 ] = \max[\, p_1, \frac{2k}{\zeta} + 2 ]$. 
We give the proof of Lemma~\ref{lemma:prelims:TomasSteinDcp} for completeness. 
\begin{proof}[Proof of Lemma~\ref{lemma:prelims:TomasSteinDcp}]
Set $N = \lceil \lambda^{1/k} \rceil$. 
Fix $p>\max[p_1, \frac{2d - (d-k)p_0}{\zeta} + p_0 ]$ and let $a$ be an element of $\ell^{2}$. 
For notational convenience we let $E$ denote the extension operator defined on sequences $a : \Z^d \to \C$ by $Ea := \FT{\Z^d}(\omega_\lambda \cdot \FT{\T^d}a) = a*\FT{\Z^d}(\omega_\lambda)$. 
We may assume that $a$ is not identically zero and by homogeneity normalize $a$ so that $\| a \|_{2} = 1$. 
We introduce a parameter $\alpha > 0$ in order to define the level sets and functions
\begin{align*}
S_\alpha = \{ \vector{\xi} \in \T^d : |Ea(\vector{\xi})| \geq \alpha \},
\qquad \text{and} \qquad
f = 1_{S_\alpha} \frac{Ea}{|Ea|}.
\end{align*}

By the Cauchy--Schwarz inequality and Birch's theorem in \cite{Birch} we have 
\begin{equation}\label{estimate:trivial}
\|Ea\|_{L^\infty} \lesssim N^{\frac{d-k}{2}}.
\end{equation} 
Therefore, we may restrict $\alpha$ to lie in the interval $[0,CN^{\frac{d-k}{2}}]$ for some positive constant $C$. 
By Parseval's identity, we have
\begin{align*}
\alpha |S_\alpha|
\leq \langle f , Ea \rangle
= \langle \FT{\T^d}f , \omega_\lambda \cdot a \rangle
= \langle \omega_\lambda \cdot \FT{\T^d}f , a \rangle.
\end{align*}
By Cauchy--Schwarz and the assumption $\| a \|_2 = 1$, it follows that
\begin{align*}
\alpha^2 |S_\alpha|^2
\leq \| (\FT{\T^d}f) \omega_\lambda \|_{\ell^2}^2
= \langle (\FT{\T^d}f) \cdot \omega_\lambda , \FT{\T^d}f \rangle.
\end{align*}
Another application of Parseval's identity implies that 
\begin{align}
\label{eq:prelims:TomasStein}
\alpha^2 |S_\alpha|^2 
\leq \langle f \ast F, f \rangle.
\end{align}

By \eqref{eq:prelims:TomasStein}, Holder's inequality and hypotheses \eqref{estimate:majorarcs} and \eqref{estimate:minorarcs} of the lemma, we have 
\begin{align*}
\alpha^2 |S_\alpha|^2 
%& \leq
%|\langle f \ast F_{\frakM} , f \rangle|
% + |\langle f \ast F_{\frakm} , f \rangle|
%\\ 
& \leq 
\| f \ast F_{\frakM} \|_{p_1} \| f \|_{p_1'} 
	+ \| f \ast F_{\frakm} \|_\infty \| f \|_1 
\\ 
& \lesssim 
N^{d-k - \frac{2d}{p_1}}  \| f \|_{p_1'}^2 
	+ \| F_{\frakm} \|_\infty \| f \|_1^2 
\\ 
&\lesssim
N^{d-k - \frac{2d}{p_1}}  | {S_\alpha} |^{\frac{2}{p_1'}} 
	+ N^{d-k-\zeta} | {S_\alpha} |^2. 
\end{align*}
Therefore, when \( \alpha \gtrsim N^{\frac{d-k}{2}-\frac{\zeta}{2}} \), we have 
\[
\alpha^2 |S_\alpha|^2 
\lesssim 
N^{d-k - \frac{2d}{p_1}} | {S_\alpha} |^{2 - \frac{2}{p_1}} 
. 
\]

\noindent Rearranging implies that 
$|S_\alpha| \lesssim \alpha^{-p_1} N^{\frac{(d-k)p_1}{2} - d}$.
Since $p > p_1$, we have 
\begin{align*}
%	&\phantom{= .} 
\int_{|Ea| \gtrsim N^{\frac{d-k}{2} - \frac{\zeta}{2}}} |Ea|^p \dm 	
& = 
p \int_{ CN^{\frac{d-k}{2} - \frac{\zeta}{2}} }^{ CN^{\frac{d-k}{2}} } \alpha^{p-1} |S_\alpha| \dalpha
\\ 
&\lesssim 
N^{\frac{(d-k)p_1}{2} - d} \int_1^{CN^{\frac{d-k}{2}}} \alpha^{p-p_1-1} \dalpha 
\\
&\lesssim 
N^{\frac{(d-k)p}{2} - d}.
\end{align*}
Altogether we have 
\begin{equation}\label{estimate:low_frequencies}
\int_{|Ea| \gtrsim N^{d/2 - \zeta/2}} |Ea|^p \dd{m}	
\lesssim 
N^{\frac{(d-k)p}{2} - d}.
\end{equation}

We are left to consider the the regime where $|Ea| \lesssim N^{\frac{d-k}{2} - \frac{\zeta}{2}}$. 
We now make use of estimate \eqref{estimate:subcritical} at the exponent $p_0$ to handle the regime where $|Ea| \lesssim N^{\frac{d-k}{2} - \frac{\zeta}{2}}$. 
This is possible by the trivial bound \eqref{estimate:trivial} as follows: 
\begin{align*}
\int_{ |Ea| \lesssim N^{\frac{d-k}{2} - \frac{\zeta}{2}} } |Ea|^p \dm
&\lesssim
(N^{\frac{d-k}{2} - \frac{\zeta}{2}})^{p - p_0} \int_{\T^r} |Ea|^{p_0} \dm
% \\ &
\lesssim_\epsilon
N^{\frac{(d-k-\zeta)(p-p_0)}{2}+ \epsilon}.
\end{align*}
Combining this estimate with \eqref{estimate:low_frequencies}, we have that 
\begin{align*}
\int |Ea|^p \dm
&= 
\int_{ |Ea| \lesssim N^{\frac{d-k}{2} - \frac{\zeta}{2}} } |Ea|^p \dm 
+ \int_{ |Ea| \gtrsim N^{\frac{d-k}{2} - \frac{\zeta}{2}} } |Ea|^p \dm
\\ &\lesssim_\epsilon
N^{\frac{(d-k)p}{2} - d} + N^{\frac{(d-k-\zeta)(p-p_0)}{2}+ \epsilon}.
\end{align*}
The latter summand is dominated by the former summand when 
$\frac{(d-k-\zeta)(p-p_0)}{2} < \frac{(d-k)p}{2} - d$. 
This is equivalent to 
\[
\frac{(d-k-\zeta)(p-p_0)}{2} 
= \frac{(d-k)p}{2} - \frac{\zeta p}{2} -\frac{(d-k-\zeta)p_0}{2} 
< \frac{(d-k)p}{2} - d
\]
which is equivalent to 
\[
\frac{\zeta p}{2}+\frac{(d-k-\zeta)p_0}{2}  > d.
\]
Rearranging this last expression we find that we need 
\[
p > \frac{2}{\zeta} (d - \frac{(d-k-\zeta)p_0}{2})
= \zeta^{-1} (2d - (d-k-\zeta)p_0)
= \frac{2d - (d-k)p_0}{\zeta} + p_0.
\]
This is precisely the range of $p>\frac{2d - (d-k)p_0}{\zeta} + p_0$.
\end{proof}

% ---
\section{Magyar's decomposition of the surface measure}\label{section:BirchMagyar}
% --

Let $\form(\vector{x}) \in \Z[\vector{x}]$ be an integral, positive definite, homogeneous form where $\vector{x} = (x_1,\dots,x_d)$. 
The heavy lifting in our theorem lies in a decomposition of Magyar for the surface measure $\omega_\lambda := {\bf 1}_{\{\vector{x} \in \Z^d : \form(\vector{x})=\lambda\}}$ where $\lambda \in \Z$; this is the counting measure on the integer points $\vector{x}$ in $\Z^d$ such that $\form(\vector{x})=\lambda$. 
To state this theorem we need to introduce a few objects. 

For $q \in \N$, $a \in U_q$ and $\vector{m} \in \Z^d$, define the normalized Birch--Weyl sums
\[
G_\form(a,q;\vector{m}) 
:= 
q^{-d}\sum_{\vector{b} \in (\Z/q)^d} \eof{\frac{a\form(\vector{b})+\vector{b}\cdot\vector{m}}{q}}.
\]
We have the bound 
\begin{equation}\label{eq:Birch:Weyl_sum}
|G_\form(a,q;\vector{m})| \lesssim_\epsilon q^{\epsilon-\kappa_\form}
%\quad \text{uniformly in} \quad \vector{m} \in \Z^d 
\quad \text{for all} \quad \epsilon>0
\end{equation}
uniformly in $a \in U_q$ and $\vector{m} \in \Z^d$ with 
\[
\kappa_{\form} := \frac{d-\dim V_\form(\C)}{2^{k-1}(k-1)}.
\]
See \cite{Magyar:ergodic} for a proof of this fact. 
The dimension $d$ is sufficiently large so that $\kappa_{\form}>2$. 

Let $d\sigma_{\form}$ denote the singular measure on $\R^d$ defined as the Gelfand--Leray form whose $\R^d$-Fourier transform is defined distributionally by the oscillatory integral 
\[
\int_{\R} \eof{t(\form(\vector{x})-1)} \; dt.
\]
It is known that 
\begin{equation}\label{equation:surfacemeasure} 
d\sigma_{\form}(\vector{x}) = dS_{\form}(\vector{x})/|\nabla \form(\vector{x})| 
\end{equation} 
where $dS_{\form}$ is the Euclidean surface area measure on the hypersurface $\{ \vector{x} \in \R^d : \form(\vector{x})=1 \}$. 
These measures are compactly supported since $\form$ is positive definite. 
%See \cite{Magyar:ergodic} for more information concerning Gelfand--Leray measures of hypersurfaces. 
We cite the following bound - see Lemma~6 on page 931 of \cite{Magyar:ergodic} - for the $\R^d$-Fourier transform of the surface measure: 
\begin{equation}\label{estimate:Birch:surface_measure}
|\widetilde{d\sigma_\form}(\vector{\xi})| 
\lesssim_\epsilon 
(1+|\vector{\xi}|)^{1-\kappa_\form+\epsilon} 
\quad \text{for each} \quad \vector{\xi} \in \R^d
\quad \text{and for all} \quad \epsilon>0.
\end{equation}

Let $\Psi$ be a $C^\infty(\R^d)$ bump function supported in the cube $[-1/8,1/8]^d$ and 1 on the cube $[-1/16,1/16]^d$ where these cubes are regarded as subsets of the torus $\T^d$. 
For each $q \in \N$, let $s$ be the integer such that $2^s \leq q < 2^{s+1}$. 
For such $q$ and for $a \in U_q$, define the Fourier multipliers 
\begin{equation*}
\mu^{a/q}_\lambda(\vector{\xi}) 
:= 
\sum_{\vector{m} \in \Z^d} G_{\form}(a,q;\vector{m}) \Psi(2^s[\vector{\xi}-\frac{\vector{m}}{q}]) \widetilde{d\sigma_{\form}}(\lambda^{\frac{1}{k}}[\vector{\xi}-\frac{\vector{m}}{q}])
\end{equation*}
for $\xi \in \T^d$. 
Generalizing work of \cite{MSW}, \cite{Magyar:ergodic} obtained a flexible decomposition of the surface measure; we choose the following form. 
\begin{Magyar}[\cite{Magyar:ergodic,MSW}]\label{Magyar:Birch}
Let $\form(\vector{x}) \in \Z[\vector{x}]$ be an regular, positive definite, homogeneous, integral form. 
For each $\lambda \in \N$ the Fourier transform of the surface measure $\omega_\lambda$ decomposes as 
\begin{equation}
\lambda^{1-\frac{d}{k}} \cdot  \FT{\Z^d}{\omega_\lambda}(\vector{\xi}) 
= 
\left( \sum_{s=0}^{\lceil \log_2 \lambda^{1/k} \rceil} \sum_{q=2^s}^{2^{s+1}-1} \eof{-\frac{a\lambda}{q}} \sum_{a \in U_q} \mu^{a/q}_\lambda(\vector{\xi}) \right) 
+ \varepsilon_\lambda(\vector{\xi})
\end{equation}
where 
\begin{equation}\label{estimate:error_bound}
\| \varepsilon_\lambda \|_{L^\infty(\T^d)} 
\lesssim_{\form,\epsilon} 
\lambda^{\epsilon-\gamma_\form} 
\quad \text{for all} \quad \epsilon>0.
\end{equation}
\end{Magyar}

\begin{remark}
Our form of the error term $\varepsilon_\lambda$ and its estimate \eqref{estimate:error_bound} do not explicitly appear in \cite{Magyar:ergodic}. 
We outline the differences and how to prove this form of Magyar's Decomposition Theorem. 
Recall that Magyar's main term takes the shape as the Fourier multiplier
\begin{equation}\label{Magyar:mainterm}
\sum_{q \in \N} \sum_{a \in U_q} \eof{-\frac{a\lambda}{q}} \sum_{\vector{m} \in \Z^d} G_{\form}(a,q;\vector{m}) \Psi(q\vector{\xi}-\vector{m}) \widetilde{d\sigma_{\form}}(\lambda^{1/k}[\vector{\xi}-\frac{\vector{m}}{q}]).
\end{equation}
The first notable difference is that we have dyadically refined the decomposition so that \eqref{Magyar:mainterm} becomes 
\begin{equation}\label{Magyar:mainterm:dyadic}
\sum_{s=0}^{\infty} \sum_{q=2^s}^{2^{s+1}-1} \sum_{a \in U_q} \eof{-\frac{a\lambda}{q}} \sum_{\vector{m} \in \Z^d} G_{\form}(a,q;\vector{m}) \Psi(2^s[\vector{\xi}-\frac{\vector{m}}{q}]) \widetilde{d\sigma_{\form}}(\lambda^{1/k}[\vector{\xi}-\frac{\vector{m}}{q}]).
\end{equation}
This modifies the analysis of (2.15) and (2.16) of Proposition~4 in \cite{Magyar:ergodic} in inconsequential ways since $2^s \leq q < 2^{s+1}$. 
In particular, this preserves the estimate \eqref{estimate:error_bound}. 
The second notable difference is that we truncated the sum over $q \in \N$. 
Following the analysis of (2.17) of Proposition~4 in \cite{Magyar:ergodic}, we may truncate \eqref{Magyar:mainterm:dyadic} to \begin{equation}\label{Magyar:mainterm:truncated}
\sum_{s=0}^{\lfloor \log_2 \lambda^{1/k} \rfloor} \sum_{q=2^s}^{2^{s+1}-1} \sum_{a \in U_q} \eof{-\frac{a\lambda}{q}} \sum_{\vector{m} \in \Z^d} G_{\form}(a,q;\vector{m}) \Psi(2^s[\vector{\xi}-\frac{\vector{m}}{q}])  \widetilde{d\sigma_{\form}}(\lambda^{1/k}[\vector{\xi}-\frac{\vector{m}}{q}])
\end{equation}
and place the difference into the error term $\varepsilon_\lambda$ while maintaining the estimate \eqref{estimate:error_bound}. 
The expert may immediately verify this by using the Magyar--Stein--Wainger transference principle (see Section~2 of \cite{MSW}) and Birch's Weyl bound \eqref{eq:Birch:Weyl_sum}. 
\end{remark}

\begin{remark}
When $\form(\vector{x}) = x_1^2+\cdots+x_d^2$ our choice of decomposition differs from that of \cite{Bourgain:lattice_point_problems} and the corresponding estimates implied by Theorem~\ref{theorem:main} are weaker than that of \cite{Bourgain:lattice_point_problems}. 
For certain diagonal forms, such as $\form(\vector{x}) := x_1^k+\cdots+x_d^k$ when $k \geq 2$, Magyar's decomposition may be refined; see \cite{MSW,ACHK} for improved bounds which yield improvements to Magyar's Decomposition Theorem for certain families of forms. 
In turn, these improvements yield improvements to Theorem~\ref{theorem:main} for those families of forms. 
\end{remark}

% \begin{proof}
% We need to give a suitable bound for  the error estimates\[
% \sum_{q\geq [\log_2{N}]} \eof{-\frac{a\lambda}{q}} \sum_{a \in U_q}\sum_{\mathbf{m}\in\Z^d}G_Q (a,q,\mathbf{m})\Psi(q\vector{\xi}-\mathbf{m}) \widetilde{d\sigma_{\form}}(\lambda^{1/k}[\vector{\xi}-\frac{\vector{m}}{q}])\]
% and \[
% \sum_{q\geq [\log_2{N}]} \eof{-\frac{a\lambda}{q}} \sum_{a \in U_q}e(-\frac{a\la}{q})G_Q (a,q,\mathbf{m})\left(\Psi(q\vector{\xi}-\mathbf{m})-\Psi(2^s(\vector{\xi}-\frac{\mathbf{m}}{q})\right)\widetilde{d\sigma_{\form}}(\lambda^{1/k}[\vector{\xi}-\frac{\vector{m}}{q}])\]
% which are $O(\lambda^{\varepsilon-\gamma_Q}).$ 

% The former estimate follows directly from \eqref{eq:Birch:Weyl_sum}. For the latter we notice that \[
% \sum_{a \in U_q}\sum_{\mathbf{m}\in\Z^d}e(-\frac{a\la}{q})G_Q (a,q,\mathbf{m})\left(\Psi(q\vector{\xi}-\mathbf{m})-\Psi(2^s(\vector{\xi}-\frac{\mathbf{m}}{q}))\right)\widetilde{d\sigma_{\form}}(\lambda^{1/k}[\vector{\xi}-\frac{\vector{m}}{q}]) \]
% has at most one summand which is non-zero. Estimates \eqref{eq:Birch:Weyl_sum} and \eqref{estimate:Birch:surface_measure} give that this summand is at most \[
% q^{1-K+\varepsilon}\frac{1}{\left(1+\lambda^{1/k}|\xi-\frac{\mathbf{m}}{q}|\right)^{K-1+\varepsilon}}\lesssim \frac{1}{q^{K-1+\varepsilon}\lambda^{(K-1+\varepsilon)/k}}2^{s(K-1+\epsilon)}.\] 
% Sum in $2^s\leq q < 2^{s+1}$ and $0\leq s \leq \log_2N+1$ and this is
% \[
% \lesssim \frac{1}{q^{K-2+\varepsilon}\lambda^{(K-1+\varepsilon)/k}}2^{s(K-1+\epsilon)}
% \lesssim \lambda^{(K-2+\varepsilon)/k}
% .
% \]
% \end{proof}

The next theorem establishes \eqref{estimate:majorarcs} of Lemma~\ref{lemma:prelims:TomasSteinDcp}; that is, we treat the major arc terms. 
\begin{theorem}\label{theorem:majorarcs}
Let $\form(\vector{x}) \in \Z[\vector{x}]$ be a positive definite, regular, integral, homogeneous form satisfying \eqref{Birch_criterion}, and $\lambda \in \N$. 
If $p > 2+\frac{4}{\kappa_\form-2}$ we have 
\begin{equation}\label{eq:main_term} 
\| F_{\major} \ast f \|_{L^{p}(\T^d)} 
\lesssim_p 
%\lambda^{\epsilon-\frac{d}{k}(\frac{2}{p}-1)} 
\lambda^{\frac{d-k}{k}-\frac{2d}{kp}} \|f\|_{L^{p'}(\T^d)}
%\quad \text{for all} \quad \Lambda \geq 1.
\end{equation}
for each $\lambda \in \N$. 
\end{theorem}

We may deduce Theorem~\ref{theorem:main} once Theorem~\ref{theorem:majorarcs} is proved as follows.  
\begin{proof}[Proof of Theorem~\ref{theorem:main} assuming Theorem~\ref{theorem:majorarcs}]
Since $k \geq 2$, we have $2k/\gamma_\form > 4/(\kappa_\form-2)$, and Lemma~\ref{lemma:prelims:TomasSteinDcp} reduces Theorem~\ref{theorem:main} to applying the major arc bound in Theorem~\ref{theorem:majorarcs} and the (minor arc) bound for the error term \eqref{estimate:error_bound}. 
\end{proof}

% ---
\section{Proof of Theorem~\ref{theorem:majorarcs}}\label{section:majorarcs}
% ---

Fix $\form(\vector{x}) \in \Z[\vector{x}]$ a positive definite, homogeneous form of degree $k$ satisfying \eqref{Birch_criterion} and $\lambda \in \N$. 
Set $N = \lceil \lambda^{1/k} \rceil$. 
%Let $\Phi(\vector{\xi}) := \Psi(\vector{\xi}) - \Psi(2\vector{\xi})$. 
% For $0 \leq j < \lfloor \log_2 N \rfloor$, define the functions $\Psi_j(\vector{\xi}) := \Psi(2^j\vector{\xi}) - \Psi(2^{j+1}\vector{\xi})$, and for $j = \lfloor \log_2 N \rfloor$, define the function $\Psi_j(\vector{\xi}) := \Psi(2^j\vector{\xi})$. 
Define the functions 
\begin{align*}
\Psi_j(\vector{\xi}) &:= \Psi(2^j\vector{\xi}) - \Psi(2^{j+1}\vector{\xi}) 
\quad \text{for} \quad 0 \leq j < \lfloor \log_2 N \rfloor, 
\; \text{and} 
\\ 
\Psi_j(\vector{\xi}) &:= \Psi(2^j\vector{\xi})
\quad \text{for} \quad j = \lfloor \log_2 N \rfloor.
\end{align*}
Furthermore, for $q \in \N, a \in U_q$ and $0 \leq j < \lfloor \log_2 N \rfloor$, define the multipliers 
\begin{equation*}
\mu^{a/q,j}_\lambda(\vector{\xi}) 
:= 
\lambda^{\frac{d}{k}-1}\sum_{\vector{m} \in \Z^d} G_{\form}(a,q;\vector{m}) \Psi_j(2^{s+1}[\vector{\xi}-\frac{\vector{m}}{q}]) \cdot \FT{\R^d}{d\sigma_{\form}}(\lambda^{1/k}[\vector{\xi}-\frac{\vector{m}}{q}]).
\end{equation*}
We will collect these multipliers according to the scale of their moduli; to do so, define, for each $s \geq 0$, the set of fractions 
\[
\mathcal{R}_s := \{ a/q \in \Q : 2^s \leq q < 2^{s+1} \; \text{and} \; a \in U_q \}. 
\]

Let $\kernel^{a/q,j}_\lambda := \FT{\T^d}{(\mu^{a/q,j}_\lambda)}$ denote the inverse Fourier transform of $\mu^{a/q,j}_\lambda$. 
We start our proof by establishing an identity for these kernels. 
\begin{proposition}\label{proposition:main_term:l1tolinfty}
Let $\form(\vector{x}) \in \Z[\vector{x}]$ be a positive definite, non-singular, integral, homogeneous form satisfying \eqref{Birch_criterion}, and $\Gamma_{\form}$ be a set of regular values for the form $\form$. 
If $s \geq 0$, then for each $a/q \in \mathcal{R}_s$ we have 
\begin{equation}\label{eq:kernel}
\kernel^{a/q,j}_\lambda(\vector{x}) 
= 
\eof{a\form(\vector{x})/q} \lambda^{-1} [\FT{\R^d}(\dilate_{2^s\lambda^{-1/k}}\Psi_j)*d\sigma_{\form}](\lambda^{-1/k} \vector{x}) 
% \quad \text{for} \quad \vector{x} \in \Z^d.
\end{equation}
for all $\vector{x} \in \Z^d.$
\end{proposition}
\noindent 
The proof of this proposition follows the proof of Proposition~1 in \cite{Hughes:improving}; in that proof, one replaces $\Psi$ by $\Psi_j$ and $q$ by $2^s$.

Now that we know the structure of our kernel, we will use a circle method decomposition and a further Littlewood--Paley decomposition to arbitrage $L^1(\T^d) \to L^\infty(\T^d)$ and $L^2(\T^d) \to L^2(\T^d)$ estimates and deduce Theorem~\ref{theorem:majorarcs}. 
These bounds are the content of the two following lemmas. 
\begin{lemma}\label{lemma:Birch:L1toLinfty}
Let $\form(\vector{x}) \in \Z[\vector{x}]$ be a positive definite, non-singular, integral, homogeneous form satisfying \eqref{Birch_criterion}, and $\lambda \in \N$. 
% If $1 \leq a < q \lesssim N$ with $(a,q)=1$ and $2^s \leq q < 2^{s+1}$, 
If $0 \leq s \leq  \lfloor \log_2 N \rfloor$ and $a/q \in \mathcal{R}_s$, 
then each major arc piece $\mu^{a/q,j}_\lambda$ satisfies 
\begin{equation}\label{estimate:L1toLinfty}
\| \mu^{a/q,j}_\lambda \|_{L^\infty(\T^d)}
\lesssim_\epsilon 
% 2^{j(\kappa-1)-s} \lambda^{\frac{d}{k}-1-\frac{\kappa-1}{k}+\epsilon}
% 2^{-\kappa s} \lambda^{\frac{d}{k}-1+\epsilon}
2^{j-s} 2^{j(\epsilon-\kappa)} \lambda^{\frac{d}{k}-\kappa}
\quad \text{for} \quad 0 \leq j \leq \lfloor \log N \rfloor -s, 
\end{equation}
and 
\begin{equation}\label{eq:L1toLinfty}
\| \mu^{a/q,j}_\lambda \|_{L^\infty(\T^d)}
\lesssim_\epsilon 
% 2^{-s\kappa} \lambda^{\frac{d}{k}-1}
%\lesssim_\epsilon
2^{s(\epsilon-\kappa)} \lambda^{\frac{d}{k}-1}
\quad \text{for} \quad j = \lfloor \log N \rfloor - s
\end{equation}
for all $\epsilon>0$.
\end{lemma}

\begin{lemma}\label{lemma:Birch:L2toL2}
Let $\form(\vector{x}) \in \Z[\vector{x}]$ be a positive definite, non-singular, integral, homogeneous form satisfying \eqref{Birch_criterion}, and $\lambda \in \N$. 
% If $1 \leq a < q \lesssim N$ with $(a,q)=1$ and $2^s \leq q < 2^{s+1}$, 
If $0 \leq s \leq  \lfloor \log_2 N \rfloor$ and $a/q \in \mathcal{R}_s$, 
then each major arc piece $\mu^{a/q,j}_\lambda$ satisfies 
\begin{equation}\label{eq:L2toL2}
\| \FT{\T^d}{\mu^{a/q,j}_\lambda} \|_{\ell^\infty(\Z^d)}
\lesssim
2^{j+s} \lambda^{-1-\frac{1}{k}}
% \lambda^{-1}
\end{equation}
for $0 \leq j \leq \lfloor \log N \rfloor - s$.
\end{lemma}

\begin{remark}
Note that $j+s = \lfloor \log_2 N \rfloor$ is the natural cut-off because when $|\vector{\xi}| \lesssim \lambda^{1/k}$ since we do not capture any oscillation in $\FT{\R^d}d\sigma(\lambda^{1/k}\vector{\xi})$ when $|\vector{\xi}| \lesssim \lambda^{1/k}$. 
\end{remark}

\begin{proof}[Proof of Lemma~\ref{lemma:Birch:L1toLinfty}]
Fix $0 \leq s \leq \lfloor \log_2 N \rfloor$ and $a/q \in \mathcal{R}_s$. 
For $0 \leq j < \lfloor \log N \rfloor -s$, \eqref{estimate:Birch:surface_measure} implies that 
\[
\| \mu^{a/q,j} \|_{L^\infty(\T^d)} 
\lesssim_\epsilon 
\lambda^{\frac{d}{k}-1} (2^s)^{\epsilon-\kappa} (\lambda^{1/k}/2^{s+j})^{1-\kappa+\epsilon}
\lesssim_\epsilon
% 2^{j(\kappa-1)-s} \lambda^{\frac{d+1-\kappa}{k}-1+\epsilon}.
2^{j-s} 2^{j(\epsilon-\kappa)} \lambda^{\frac{d}{k}-\kappa}
\]
for all $\epsilon>0$ since $\kappa>2$. 
% We rewrite the last expression as 
% \[
% 2^{(j+s)(\kappa-1)-s(\kappa-1)-s} \lambda^{\frac{d+1-\kappa}{k}-1+\epsilon} 
% = 
% 2^{(j+s)(\kappa-1)-s\kappa} \lambda^{\frac{d+1-\kappa}{k}-1+\epsilon}.
% \]
% Since $j+s \leq \lfloor \log_2 N \rfloor$, we find that $2^{(j+s)(\kappa-1)} \lesssim \lambda^{\frac{\kappa-1}{k}}$, and therefore, 
% \[
% \| \mu^{a/q,j} \|_{L^\infty(\T^d)} 
% \lesssim_\epsilon
% 2^{-s\kappa} \lambda^{\frac{d}{k}-1+\epsilon}.
% \]
For $j = \lfloor \log N \rfloor - s$, \eqref{estimate:Birch:surface_measure} implies that 
\[
\| \mu^{a/q,j} \|_{L^\infty(\T^d)} 
\lesssim_\epsilon
2^{s(\epsilon-\kappa)} \lambda^{\frac{d}{k}-1}.
% \lesssim_\epsilon
% 2^{-s\kappa} \lambda^{\frac{d}{k}-1+\epsilon}.
\]
\end{proof}

% \begin{comment}[KH 2019.12.29]
% The above handled a single piece $\mu^{a/q,j}$. 
% Unfortunately, the pieces $\mu^{a/q,j}$ overlap as $a/q$ varies in $\mathcal{R}_s$. 
% To see this, consider $\vector{\xi}=0$. 
% Consequently, we cannot extend this bound to $\sum_{a/q \in \mathcal{R}_s} \mu^{a/q,j}$. 
% \end{comment}

Before proving Lemma~\ref{lemma:Birch:L2toL2}, we need a geometric property of our measures $d\sigma_\form$. 
The estimate below is best known for $\form(\vector{x}) = |\vector{x}|^2$; see \cite{Grafakos} for this estimate. 
However, we are unaware of a reference for more general hypersurfaces aside from estimate (23) in \cite{Hughes:improving}. 
For completeness, we include the statement and its proof below. 
\begin{proposition}\label{prop:neighborhood}
Let $\phi$ be a Schwartz function on $\R^d$. 
If $t>0$, then 
\begin{equation}\label{estimate:neighborhood}
\| t^{-d}\dilate_t\phi * d\sigma_Q \|_{L^\infty(\R^d)} 
\lesssim t^{-1}. 
\end{equation}
\end{proposition}

\begin{proof}
Since $\form$ is positive definite, the variety $V_\form(\R)$ is compact. 
Moreover, \eqref{equation:surfacemeasure} implies that for every ball $B$ of radius $r>0$ we have 
\begin{equation}\label{estimate:balls}
\sigma(B) \lesssim r^{d-1}.
\end{equation}
For each point $\vector{x} \in \R^d$, define the sets $S_0(\vector{x}) := \{ \vector{y} \in \R^d : |\vector{x}-\vector{y}| < t \}$ and $S_j(\vector{x}) := \{ \vector{y} \in \R^d : 2^jt \leq |\vector{x}-\vector{y}| < 2^{j+1}t \}$ for $j \in \N$. 
By \eqref{estimate:balls} we have that 
\begin{equation}\label{estimate:smoothballs}
\sigma(S_j(\vector{x})) \lesssim (2^jt)^{d-1} 
\end{equation} 
for each $\vector{x} \in \R^d$. 

Since $\phi$ is Schwartz, we have 
\[
\dilate_t\phi(\vector{x}) 
\lesssim_\phi 
(1+|\vector{x}/t|)^{-M}
\]
for all $M \in \N$. 
Therefore, 
\[
\dilate_t\phi * d\sigma_\form(\vector{x})
\lesssim 
(1+|\cdot/t|)^{-M} * d\sigma_\form(\vector{x})
\]
for all $\vector{x} \in \R^d$. 
Decomposing $\R^d$ into the sets $S_j(\vector{x})$ we have 
\begin{align*}
\dilate_t\phi * d\sigma_\form(\vector{x})
& \lesssim_{\phi,M} 
% (1+|\cdot/t|)^{-M} * d\sigma_\form(\vector{x})
% \\ & \lesssim 
\int_{\R^d} (1+|\vector{x}-\vector{y}|/t)^{-M} \; d\sigma(\vector{y})
\\ & \lesssim \sum_{j=0}^\infty \int_{S_j(\vector{x})} (1+|\vector{y}|/t)^{-M} \; d\sigma(\vector{y})
\\ & \lesssim \sum_{j=0}^\infty \int_{S_j(\vector{x})} 2^{-jM} \; d\sigma(\vector{y})
% \\ & \lesssim \sum_{j=0}^\infty \sigma_\form(S_j(\vector{x})) 2^{-jM}
% \\ & \lesssim \sum_{j=0}^\infty (2^jt)^{d-1} 2^{-jM}
% \\ & \lesssim t^{d-1}.
\end{align*}
Using estimate \eqref{estimate:smoothballs} we obtain that 
\begin{align*}
\dilate_t\phi * d\sigma_\form(\vector{x})
& \lesssim_{\phi,M} 
\lesssim \sum_{j=0}^\infty \sigma_\form(S_j(\vector{x})) 2^{-jM}
\lesssim \sum_{j=0}^\infty (2^jt)^{d-1} 2^{-jM}
\lesssim t^{d-1}.
\end{align*}
Normalizing by $t^{-d}$ we obtain the desired estimate. 
\end{proof}

\begin{proof}[Proof of Lemma~\ref{lemma:Birch:L2toL2}]
Fix $0 \leq s \leq \lfloor \log_2 N \rfloor$ and $a/q \in \mathcal{R}_s$. 
For each $0 \leq j \leq \lfloor \log N \rfloor - s$, identity \eqref{eq:kernel} and estimate \eqref{estimate:neighborhood} imply that for each $\vector{x} \in \Z^d$ we have 
\begin{align*}
\mu^{a/q,j}_\lambda(\vector{x}) 
& \lesssim_d 
2^j (\lambda^{1/k}/2^s)^{-1} \lambda^{-1} 
% (1+|\lambda^{-1/k}\vector{x}|)^{-2d} 
% \\ & 
\lesssim_d 
2^{j+s} \lambda^{-1-\frac{1}{k}}
\end{align*}
by taking $\phi = \FT{\R^d}(\dilate{2^s\lambda^{-1/k}}\psi_j)$ and $t = \lambda^{1/k}2^{-s}$ in Proposition~\ref{prop:neighborhood}. 
% Since $0 \leq j+s \leq \lfloor \log_2 N \rfloor$, we have \eqref{eq:L2toL2}. 
% \[
% \| \mu^{a/q,j}_\lambda \|_{\ell^\infty(\Z^d)}
% \lesssim_d 
% \lambda^{1/k} \lambda^{-1-\frac{1}{k}} 
% = 
% \lambda^{-1}.
% \]
\end{proof}

\begin{proof}[Proof of Theorem~\ref{theorem:majorarcs}]
Let $1 \leq p \leq 2$ and $f \in L^{p'}(\T^d)$ be normalized so that $\|f\|_{L^{p'}(\T^d)}=1$.
Interpolating the bounds \eqref{estimate:L1toLinfty} and \eqref{eq:L2toL2} for $\mu^{a/q,j}_\lambda$ when $0 \leq j+s < \lfloor \log_2 N \rfloor$, we obtain 
\begin{align*}
\| \mu^{a/q,j}_\lambda \ast f \|_{p} 
& \lesssim_\epsilon 
\left( 2^{j+s} \lambda^{-1-\frac{1}{k}} \right)^{\frac{2}{p}} \cdot \left( 2^{j-s} 2^{j(\epsilon-\kappa)} \lambda^{\frac{d}{k}-\kappa} \right)^{1-\frac{2}{p}} 
\\ & = 
2^{j(\frac{2}{p}+(1+\epsilon-\kappa)(1-\frac{2}{p}))} \cdot 2^{s(\frac{2}{p}-1+\frac{2}{p})} \cdot \lambda^{(\frac{d}{k}-\kappa)(1-\frac{2}{p})-\frac{2}{p}(1+\frac{1}{k})}
\\ & = 
2^{j(1+(\epsilon-\kappa)(1-\frac{2}{p}))} \cdot 2^{s(\frac{4}{p}-1)} \cdot \lambda^{(\frac{d}{k}-\kappa)(1-\frac{2}{p})-\frac{2}{p}(1+\frac{1}{k})}
.
\end{align*}
Summing over fractions $a/q \in R_s$ for $j \leq s < \lfloor \log_2 N \rfloor$, we find that 
\[
\left\| \left( \sum_{a/q \in R_s} \mu^{a/q,j}_\lambda(\vector{x}) \right) \ast f \right\|_{L^p(\T^d)}
\lesssim_{\form,\epsilon} 
2^{j(1+(\epsilon-\kappa)(1-\frac{2}{p}))} \cdot 2^{s(\frac{4}{p}+1)} \cdot \lambda^{(\frac{d}{k}-\kappa)(1-\frac{2}{p})-\frac{2}{p}(1+\frac{1}{k})}.
\]
Provided $1-\kappa(1-\frac{2}{p}) < 0$, which is equivalent to the range $p>2+\frac{2}{\kappa-1}$, we have 
\[
\left\| \left( \sum_{j=0}^{\lfloor \log_2N \rfloor - s-1} \sum_{a/q \in R_s} \mu^{a/q,j}_\lambda(\vector{x}) \right) \ast f \right\|_{L^p(\T^d)}
\lesssim_{\form,\epsilon} 
2^{s(\frac{4}{p}+1)} \cdot \lambda^{(\frac{d}{k}-\kappa)(1-\frac{2}{p})-\frac{2}{p}(1+\frac{1}{k})}.
\]
Consequently, when $p>2+\frac{2}{\kappa-1}$, we have 
\[
\left\| \left( \sum_{s=0}^{\lfloor \log_2N \rfloor} \sum_{j=0}^{\lfloor \log_2N \rfloor - s-1} \sum_{a/q \in R_s} \mu^{a/q,j}_\lambda(\vector{x}) \right) \ast f \right\|_{L^p(\T^d)}
\lesssim_{\form,p} 
\lambda^{(\frac{4}{p}+1)/k} \cdot \lambda^{(\frac{d}{k}-\kappa)(1-\frac{2}{p})-\frac{2}{p}(1+\frac{1}{k})}
= 
\]
Comparing the exponent of $\lambda$ with the desired one of $\frac{d}{k}-1-\frac{2d}{kp}$ we find that we have \eqref{theorem:majorarcs} for $p>2+\frac{4}{k\kappa-\kappa-1}$. 
This is better than the range of $p>2+\frac{4}{\kappa-2}$ claimed in the theorem. 

When $0 \leq j+s = \lfloor \log_2 N \rfloor$, we have 
\[
\| \mu^{a/q,j}_\lambda \ast f \|_{p} 
\lesssim_\epsilon 
\left( \lambda^{-1} \right)^{\frac{2}{p}} \cdot \left( 2^{s(\epsilon-\kappa)} \lambda^{\frac{d}{k}-1} \right)^{1-\frac{2}{p}} 
= 
2^{s(\epsilon-\kappa)(1-\frac{2}{p})} \cdot \lambda^{\frac{d}{k}-1-\frac{2d}{kp}}
.
\]
Summing over $0 \leq s \leq \lfloor \log_2 N \rfloor$, we find that 
\[
\| \sum_{s=0}^{\lfloor\log_2N\rfloor} \mu^{a/q,j}_\lambda \ast f \|_{p} 
\lesssim 
\lambda^{\frac{d}{k}-1-\frac{2d}{kp}}
\]
provided that $(\epsilon-\kappa)(1-\frac{2}{p}) < 0$ for arbitrarily small, positive $\epsilon$. 
For each $0<\epsilon<\kappa-2$, this is equivalent to the range of $p>\frac{2(\kappa-\epsilon)}{\kappa-2-\epsilon}$. 
Thereby taking $\epsilon$ to 0, we arrive at the range of $p>\frac{2\kappa}{\kappa-2} = 2+\frac{4}{\kappa-2}$, as claimed. \end{proof}

% 
% -----
% \newpage

\bibliographystyle{amsalpha}
% \bibliographystyle{plain}
%\bibliography{references_lp_improving}

\end{document}